\documentclass[11pt]{amsart}
\usepackage[utf8]{inputenc}

\usepackage{amssymb} % Contains \mathbb, \mathcal, and \varnothing
\usepackage{amsthm} % Allows the use of theorems
\usepackage{thmtools}
\usepackage{amsmath} % Allows theorem numbering within a section
\usepackage{subfiles} % Calls separate files into the main document
\usepackage{fancyhdr} % Used to display header line and text
\usepackage{bbm} % Allows usage of Blackboard font numbers 1 and 2
\usepackage{xcolor} % Allows font color to be changed
\usepackage[hidelinks]{hyperref} % Allows colored clickable links

\usepackage{mdframed} % Allows boxes to be drawn
\usepackage{enumitem} % Allows enumerate to have custom "bullet points"
\usepackage{relsize} % Contains \mathlarger
\usepackage{lastpage} % Can reference the last page of the document
\usepackage{graphicx}
\usepackage{caption}
\usepackage{wrapfig}
\usepackage{quiver}
\usepackage{multicol}
\usepackage{cleveref}

\usepackage{paracol}
\usepackage[english]{babel}
\usepackage{mathrsfs}
\usepackage{cancel}
\usepackage[hidelinks]{hyperref}
\usepackage[numbers]{natbib}
\usepackage{doi}
\usepackage{url}

\newcommand{\F}{\mathbb{F}}

\newcommand{\K}{\mathbb{K}}
\renewcommand{\L}{\mathbb{L}}

\newcommand{\N}{\mathbb{N}} % Natural numbers

\newcommand{\Q}{\mathbb{Q}} % Rational numbers
\newcommand{\Z}{\mathbb{Z}} % Integers
\newcommand{\1}{\mathbbm{1}} % Identity permutation

\newcommand{\Fq}{\mathbb{F}_q}

\newcommand{\HH}{\mathcal{H}} % Upper half-plane

\newcommand{\LL}{\mathcal{L}} % Laplace transform

\newcommand{\NN}{\mathcal{N}}
\newcommand{\qlb}{\overline{\Q}_\ell}

\newcommand{\fqb}{\overline{\F}_q}

\newcommand{\hl}{(H,\LL)}

\newcommand{\dg}{\mathscr{D}(G)} %derived category G
\newcommand{\dx}{\mathscr{D}(X)} %derived category X
\newcommand{\dgg}{\mathscr{D}_G(G)} %equivariant derived category of G
\newcommand{\dggp}{\mathscr{D}_{G'}(G')} %equivariant derived category of G'
\newcommand{\dgx}{\mathscr{D}_G(X)} %equivariant derived category of G on X
\DeclareMathOperator{\ind}{ind}
\DeclareMathOperator{\id}{id}
\DeclareMathOperator{\Hom}{Hom}
\DeclareMathOperator{\pr}{pr}
\renewcommand{\url}[2]{\href{#1}{\textcolor{blue}{#2}}} % Clickable blue link
\newtheorem{theorem}{Theorem}
\numberwithin{theorem}{section}
\newtheorem{corollary}[theorem]{Corollary}
\newtheorem{lemma}[theorem]{Lemma}
\newtheorem{prop}[theorem]{Proposition}
\newtheorem*{conjecture}{Conjecture}

\newtheorem{mainthm}{Theorem}

\theoremstyle{remark}
\newtheorem*{remark}{Remark}

\theoremstyle{definition}
\newtheorem{definition}[theorem]{Definition}
\newtheorem*{example}{Examples}
\newtheorem*{nonex}{Non-example}

\makeatletter
\def\l@subsection{\@tocline{2}{0pt}{2.5pc}{5pc}{}}
\makeatother

\title{Unipotent groups with trivial $\L$-packets \\ are easy}
\author{Sadie Lipman}
\address{Department of Mathematics, University of Michigan, Ann Arbor, Michigan}
\email{slipman@umich.edu}
\date{April 23, 2026}
\thanks{This work was supported by NSF-DMS-2507946, NSF-DMS-1840234, and a Simons Dissertation Fellowship.}
\begin{document}

\bibliographystyle{alphaurl}

\maketitle

\begin{abstract}
    In 2006, Boyarchenko and Drinfeld conjectured that for a unipotent algebraic group over a field of positive characteristic, every geometric point is contained in the neutral connected component of its centralizer if and only if its $\L$-packets of character sheaves are singletons. In 2013, Boyarchenko proved the ``only if" direction for $\fqb$. In this paper, we complete the proof of the conjecture in this case. Along the way, we explore the relationship between general algebraic groups satisfying this property and their Asai twisting operator.
\end{abstract}
\section*{Introduction}
The theory of character sheaves for unipotent groups was first conjectured by George Lusztig in \cite{lusztig2006}. This theory was later developed by Mitya Boyarchenko and Vladimir Drinfeld in \cite{foundations}, \cite{characters}, and \cite{chsheaves}. In 2006, Boyarchenko and Drinfeld proposed seven initial conjectures in the development of this theory (\cite{motivatedintro}). Six of the seven were proved in \cite{foundations} and \cite{chsheaves}. The goal of this paper is to complete the proof of the last remaining conjecture in the case $k=\fqb$:
 
\begin{conjecture}[{\cite[Conjecture 5]{motivatedintro}}]\label{conj}
    Let $G$ be a unipotent group over a field $k$ of positive characteristic. $G$ is easy if and only if it has trivial $\L$-packets.
\end{conjecture}

 In \cite{chsheaves}, Boyarchenko proved the ``only if" direction when $k=\fqb$. The following theorem completes the proof of the conjecture in this case:
 
\begin{mainthm}[\ref{thm}]
    If $G$ is a unipotent group over $\fqb$ with trivial $\L$-packets, then $G$ is easy.
\end{mainthm}

The proof relies on the relationship between character sheaves and Shintani descent described by Deshpande in \cite{unipshint}.  In the process, we prove a statement about general algebraic groups:

\begin{mainthm} [\ref{asai prop}]  Let $G$ be a connected algebraic group over $\fqb$ with a Frobenius $F$. The Asai twisting operator is trivial on $C(G^{F^m}/\sim)$ for all $m \in \Z_{>0}$ if and only if $G$ is easy.
\end{mainthm}

\subsection*{Acknowledgments} 
I thank Tanmay Deshpande for generously sharing his expertise and ideas. I am grateful to my advisor, Charlotte Chan, for introducing me to the subject area, many helpful meetings, and comments on a previous draft. I also thank Sydney Mathematical Research Institute for excellent working conditions.

\section{Preliminaries}
\subsection*{Notation} Let $G$ be an algebraic group over $k=\fqb$ with $q=p^n$. Fix a prime $\ell \neq p$. We define an algebraic group to be a smooth group scheme of finite type over $k$. A unipotent algebraic group (often just ``unipotent group" for brevity) is an algebraic group over $k$ that is isomorphic to a closed subscheme of $\mathrm{UL}_n$, the group of unipotent $n \times n$ upper triangular matrices, for some $n$.

\subsection{Easy algebraic groups}
\begin{definition}
    We say an algebraic group is \emph{easy} if for every $g \in G(k)$, we have $g$ is in the neutral connected component of its centralizer $Z(g)$ in $G$.
\end{definition}

This definition can be extended to algebraic groups over non-algebraically closed fields by checking this condition on geometric points after base change to the algebraic closure. It is clear that all easy groups must be connected as for any $g \in G(k)$ we have $g \in Z(g)^\circ \subset G^\circ$.

\begin{example}

        $\mathrm{GL}_n$ is an easy algebraic group. Over a field of characteristic zero, we have that any reductive group is easy if and only if its derived group is simply connected and its center is connected.

        Any unipotent group over a field of characteristic zero is easy. This follows from the fact that every unipotent group in characteristic zero is connected and each closed subgroup of a unipotent group is again a unipotent group. However, it is not the case that every unipotent group over a field of characteristic $p$ is easy.

        The simplest example of any easy unipotent group over $k$ is $\mathrm{UL}_n$. In general, if $U$ is the maximal unipotent subgroup of a reductive group $G$ over $k$, then $U$ is easy given the characteristic of $k$ is sufficiently large.

\end{example}

\begin{nonex}
    Let $G$ be a noncommutative connected unipotent group of dimension $2$. Then, $G$ is not easy when the characteristic of $k$ is larger than 2. The only nontrivial connected subgroup of $G$ is $Z(G)$. If $g \in G(k)\setminus Z(G)(k)$, then $Z(g)^\circ \neq G$. Thus, $g \notin Z(G)=Z(g)^\circ$.
\end{nonex}

\subsection{The Asai twisting operator} \label{asai def}

We recall some of the constructions given by Shoji in \cite{shintani}. Let $F: G \to G$ be a Frobenius map and denote the $F$-fixed points by $G^F$. Let $G^F/ \sim$ denote the conjugacy classes of $G^F$.

Given a connected algebraic group $G$, we can use Lang's theorem to define a bijective map
$$\mathrm{N}_{F/F}: (G^{F}/\sim) \to (G^{F}/\sim)$$
$$g=xF(x)^{-1} \mapsto F(x)^{-1}x.$$

We define the \emph{Asai twisting operator} to be the map:
$$\Theta_F:=N_{F/F}^*: C(G^{F}/\sim) \to C(G^{F}/\sim).$$

 This map is not in general the identity. \Cref{twisted cent} and \Cref{asai prop} tell us for which connected algebraic groups $\Theta_F$ acts as the identity.

\begin{remark}
    The Asai twisting operator is a special case of the $m$-th Shintani map with $m=1$.
\end{remark}

\subsection{Character sheaves for unipotent groups}
For the remainder of Section 1, $G$ is a unipotent group. Let $\mathscr{D}(X):=D_c^b(X, \qlb)$. Given an action of $G$ on $X$, we can define the equivariant derived category, which we will denote by $\mathscr{D}_G(X)$. Let $\dgg$ be the equivariant derived category, where $G$ acts on itself by conjugation.

For unipotent groups we can use the following definition of $\mathscr{D}_G(X)$:

Let $\alpha: G \times X \rightarrow X$ be the action morphism and let $\mathrm{pr}_2: G \times X \rightarrow X$ be the projection. Denote the multiplication morphism by $\mu: G \times G \rightarrow G$. Let $\mathrm{pr}_{23}: G \times G \times X \rightarrow G \times X$ be the projection onto the second and third factors. 
\begin{definition}\label{dgx def}
    An object of the category $\dgx$ is a pair $(M, \phi)$ , where $M \in \dx$ and $\phi: \alpha^*M \xrightarrow{\sim} \pr_2^*M$ is an isomorphism in $\mathscr{D}(G\times X)$ such that
    $$\pr_{23}^*(\phi) \circ(\id_G \times \alpha)^* (\phi) = (\mu \times \id_X)^*(\phi),$$
    i.e., 
    $$(\id_G \times \alpha)^*\alpha^*M \cong (\mu \times \id_X)^*\alpha^*M \xrightarrow{(\mu \times \id_X)^*(\phi)} (\mu \times \id_X)^*\pr_2^*M \cong \pr_{23}^*\pr_2^*M$$
    equals the composition
    $$(\id_G \times \alpha)^*\alpha^* M \xrightarrow{(\id_G \times \alpha)^*(\phi)}(\id_G \times \alpha)^* \pr_2^*M \cong \pr_{23}^* \alpha^*M \xrightarrow{\pr_{23}^*(\phi)} \pr_{23}^*\pr_2^*M.$$
    A morphism $(M, \phi) \rightarrow (N, \psi)$ in $\dgx$ is a morphism $\nu: M \rightarrow N$ in $\dx$ satisfying $\phi \circ \alpha^*(\nu) = \pr_2^*(\nu) \circ \psi$. The composition of morphisms in $\dgx$ is defined to be equal to their composition in $\dx$.
\end{definition}

\begin{remark}
    This definition does not hold for a general algebraic group $G$. To see that for unipotent groups this definition is equivalent to the standard definition, see \cite[Appendix C]{foundations}.
\end{remark}

\begin{definition}
    Let $M, N \in \dg$. Let $\mu: G \times G \rightarrow G$ be the multiplication morphism, and let $\pr_1,\pr_2: G \times G \rightarrow G$ denote the first and second projections, respectively. Then, we define \emph{convolution with compact supports} of $M$ and $N$ to be $M * N := \mu_!(\pr_1^*M \otimes \pr_2^*N)$. Going forward, convolution will always mean convolution with compact supports.
\end{definition}
\begin{definition}
    We say $e \in \dgg$ is a \emph{weak idempotent} if $e * e \cong e$. We say $e$ is \emph{minimal}, if for all weak idempotents $f$, we have $e * f \cong e$ or $0$.
\end{definition}
\begin{definition}
    We say $e \in \dgg$ is a \emph{closed idempotent} if there exists an arrow $\pi: \1 \rightarrow e$ that becomes an isomorphism after convolution with $e$. Such an $e$ is \emph{minimal} if for all closed idempotents $f \in \dgg$, we have $e * f \cong e$ or $0$.
\end{definition}

\begin{definition}
   Let $e \in \dgg$ be a minimal closed idempotent. Define the associated Hecke subcategory $e\dgg:=\{M \in \dgg | e * M \cong M\}$. Let $\mathscr{M}_e^{perv}$ denote the subcategory consisting of all $M\in e\dgg$, such that $M$ is perverse after forgetting its $G$-equivariant structure. 
   
   The \emph{$\L$-packet of character sheaves} corresponding to $e$, which we will denote by $\L(e)$, are the indecomposable objects in $\mathscr{M}_e^{perv}$. We say an object $M \in \dgg$ is a \emph{character sheaf} if it lies in the $\L$-packet corresponding to $e$ for some minimal closed idempotent $e$. 
\end{definition}

\subsection{Minimal closed idempotents and admissible pairs}
We recall contructions of minimal closed idempotents and results about their corresponding $\L$-packets given in \cite{foundations}. This construction relies on the notion of an \emph{admissible pair}.

\begin{definition}
    We say $\LL\in \dgg$ is a \emph{multiplicative local system} if $\LL$ is a rank one $\qlb$-local system on $G$ such that $\mu^*\LL \cong \LL \boxtimes \LL$, where $\mu: G \times G \rightarrow G$ is the multiplication morphism. 
\end{definition}

\begin{definition}
    Let $G$ be a unipotent group over $k$, and let $\hl$ be a pair consisting of a connected subgroup $H$ of $G$ and a multiplicative local system $\LL$ on $H$ . The \emph{normalizer} $N_G\hl$ of $\hl$ in $G$ is defined to be the stabilizer of the isomorphism class $[\LL] \in H^*(k)$ in the normalizer $N_G(H)$ of $H$ in $G$.
\end{definition}
\begin{definition}
    Let $G$ be a unipotent group over $k$. An \emph{admissible pair} for $G$ is a pair $(H,\LL)$ consisting of a connected subgroup $H \subset G$ and a multiplicative local system $\LL$ on $H$ such that the following three conditions hold:
    \begin{itemize}
        \item[(1)] Let $G'$ be the normalizer of $(H,\LL)$ in $G$, and let $G'^\circ$ denote its neutral connected component. Then $G'^\circ /H$ is commutative.
        \item[(2)] The $k$-group morphism $\varphi_\LL: (G'^\circ /H)_{perf} \rightarrow (G'^\circ / H)_{perf}^*$ is an isogeny, where $\varphi_\LL$ is constructed in \S 3.3 of \cite{foundations}.
        \item[(3)] For every $g\in G(k) \setminus G'(k)$, we have 
        $$\LL |_{(H\cap H^g)^\circ} \not\cong \LL^g|_{(H\cap H^g)^\circ},$$
        where $H^g = g^{-1}Hg$ and $\LL^g$ is the multiplicative local system on $H^g$ obtained from transport of structure via the map $[h \mapsto g^{-1}hg]$.
    \end{itemize}
\end{definition}

\begin{remark}
    This paper relies on many results from Boyarchenko and Drinfeld's paper constructing the theory of character sheaves. The results about character sheaves and the category $\dgg$ are often stated for perfect unipotent groups $G$. Since we have an equivalence between the \'etale topos of a unipotent group and that of its perfectization, these results automatically hold for arbitrary unipotent groups. See \cite[Remark 1.26]{foundations} for more details. As such, when stating these results going forward, we remove the assumption of perfect.
\end{remark}

Theorem 1.15 in \cite{foundations} gives us a relationship between $\mathscr{M}_e^{perv}$ and the category $e\dgg$. This gives us a way to describe objects in $\dgg$ using the structure of the $\L$-packets in $G$.

\begin{theorem}[{\cite[Theorem 1.15]{foundations}}]\label{edgg cat thm}
    Let $G$ be a unipotent algebraic group over $k$, and let $e\in \dgg$ be a minimal closed idempotent.
    \begin{itemize}
        \item[(a)]$\mathscr{M}_e^{perv}$ is a semisimple abelian category with finitely many simple objects. %In particular, $\L$-packets of character sheaves on $G$ are finite.
        \item[(b)]There exists a (necessarily unique) integer $n_e$ such that $e[-n_e] \in \mathscr{M}_e^{perv}$. One has $0\le n_e \le \dim G$. The subcategory $\mathscr{M}_e := \mathscr{M}_e^{perv}[n_e]$ of the monoidal category $e\dgg$ is monoidal.
        \item[(c)]The perverse t-structure on $\dg$ induces a t-structure on $e\dgg$, and the canonical functor $D^b(\mathscr{M}_e^{perv}) \rightarrow e\dgg$ is an equivalence.
    \end{itemize}
\end{theorem}

In particular, this tells us that $\L$-packets are finite and nonempty. Moreover, in the case that $G$ has trivial $\L$-packets, we have an explicit description of all character sheaves.

We can use admissible pairs to construct minimal closed idempotents in the following way: let $G$ be a unipotent group over $k$, and let $\hl$ be an admissible pair for $G$ with normalizer $G'$. Let $\K_H \cong \qlb [2\dim H]$ be the dualizing complex of $H$. Let $e_{\LL}=\LL \otimes \K_H$. Since $\hl$ is invariant under the conjugation action of $G'$, the complexes $\LL$ and $\K_H$ have $G'$-equivariant structures. Thus, $e'_{\LL} := \iota_! e_\LL$ lies in $\dggp$, where $\iota_!$ denotes the extension by zero functor. 

\begin{lemma}
    $e'_\LL \in \dggp$ is a closed idempotent, and is minimal as a weak idempotent.
\end{lemma}
The next theorem gives us the existence of minimal closed idempotents:
\begin{theorem}[{\cite[Theorem 1.49]{foundations}}]
    Let $G$ be a unipotent group over $k$.
    \begin{itemize}
        \item[(a)]Every minimal closed idempotent in $\dgg$ is minimal as a weak idempotent.
        \item[(b)]Every minimal weak idempotent in $\dgg$ is closed.
        \item[(c)]For every nonzero object $M \in \dg$, there exists a minimal closed idempotent $f\in \dgg$ with $f*M \ne 0$. 
    \end{itemize}
\end{theorem}
\begin{remark}
    Although we have that the minimal weak idempotents are the same and the minimal closed idempotents, it is not the case in general that weak idempotents are the same closed idempotents.
\end{remark}
The following theorem gives a classification of all minimal closed idempotents in $\dgg$ using admissible pairs of $G$:
\begin{theorem}[{\cite[Theorem 1.41]{foundations}}]
    Let $G$ be a unipotent group over $k$. 
    \begin{itemize}
        \item[(a)] Let $\hl$ be an admissible pair for $G$ with normalizer $G'$. Let $e'_\LL \in \dggp$ be the corresponding minimal idempotent. Then $f \in \ind_{G'}^G e'_\LL$ is a minimal closed idempotent in $\dgg$.
        \item[(b)] Every minimal closed idempotent $f \in \dgg$ arises from some admissible pair $\hl$ by means of the construction in part (a).
    \end{itemize}
\end{theorem}

In \cite{foundations}, Boyarchenko and Drinfeld develop a geometric reduction process for admissible pairs in $G$, which allows us to construct admissible pairs and closed idempotents from connected subgroups of $G$. We make use of this in the construction of particular closed idempotents to prove connectedness of unipotent groups with trivial $\L$-packets.

\begin{definition}
    If $G$ is an algebraic group over $k$, we define $\mathscr{P}(G)$ to be the set of pairs $\hl$ where $H$ is a connected subgroup of $G$ and $\LL$ is a multiplicative local system on $H$. We write $\mathscr{P}_{norm}(G) \subset \mathscr{P}(G)$ to be the subset of pairs $\hl \in \mathscr{P}(G)$ such that $H$ is normal in $G$
\end{definition}

\begin{definition}
    We can define a partial order on $\mathscr{P}(G)$ by $(H_1, \LL_1) \le (H_2, \LL_2)$ if $H_1 \subset H_2$ and $\LL_2 |_{H_1} \cong \LL_1$.
\end{definition}

\begin{definition}
    Given $M \in \dg$, $M \ne 0$. We say $(A, \NN) \in \mathscr{P}(G)$ is \emph{compatible with $M$} if $\iota^{A\subset G}_! \NN * M \ne 0$, where $\iota^{A \subset G}: A \hookrightarrow G$ is the inclusion. 
\end{definition}

\begin{remark}
    For any $M \in \dgg$ with $M\ne 0$, there exists a pair $\hl \in \mathscr{P}_{norm}(G)$ that is compatible with $M$. For example, take $H = \{1\}$.
\end{remark}

The next proposition allows us to construct admissible pairs from this partial order on $\mathscr{P}(G)$:

\begin{prop}[{\cite[Proposition 4.4]{foundations}}]\label{ad pair prop}
    Let $G$ be a unipotent group over $k$ and $M \in \dg$, $M \ne 0$. Suppose that $\hl \in \mathscr{P}_{norm}(G)$ is maximal among all pairs $(A,\NN) \in \mathscr{P}_{norm}(G)$ that are compatible with $M$. If $\LL$ is invariant under the conjugation action of $G$, then the pair $(H,\LL)$ is admissible for $G$.
\end{prop}

\subsection{Twists in the category $\dgg$}
The category $(\dgg, *)$ has a natural ribbon structure. We consider  unipotent groups $G$, so we can use Definition \ref{dgx def} to define the ribbon structure. In particular, twists in $\dgg$ play an important role in our proof of Theorem \ref{thm}. We will only define twists; for a complete description of the ribbon structure of $\dgg$, see \S A.5 in \cite{foundations}.

\begin{definition}[twists on $\dgg$]
    Let $c:G \times G \rightarrow G$ be the conjugation morphism $c(g,h)=ghg^{-1}$, let $\pr_2: G \times G \rightarrow G$ denote the second projection, and $\Delta: G \times G \rightarrow G$ the diagonal morphism. Then $c \circ \Delta = \id_G = \pr_2 \circ \Delta$. For each $M \in \dgg$, the $G$-equivariant structure on $M$ yields and isomorphism $\pr_2^*M \xrightarrow{\sim} c^*M$. Pulling this back by $\Delta$, we get an isomorphism $\theta_M = M= \Delta^*\pr_2^*M \xrightarrow{\sim} \Delta^*c^*M=M$. We call $\theta_M$, the \emph{twist of automorphism of $M$}.

    Let $\theta$ be the collection $\{\theta_M | M \in \dgg\}$. Then, $\theta$ is an automorphism of the identity functor in $\dgg$. We call $\theta$ the \emph{twist automorphism}.
\end{definition}

Let $e\in\dgg$ be a closed idempotent. The Hecke subcategory $(e\dgg, *)$ inherits a ribbon structure from $(\dgg, *)$. See \S A.6 in \cite{foundations} for a general description of this for Hecke subcategories of Grothendieck-Verdier categories. We use Lemma A.53 in \cite{foundations}, which applied to our setting gives: 

\begin{lemma}
    \begin{itemize}
        \item[(a)] Suppose $\psi$ is a pivotal structure on $\dgg$ and $\tilde \psi$ is the induced pivotal structure on $e\dgg$ (see Lemma A.51 in \cite{foundations}). If $ \theta$ and $\tilde \theta$ are the twists on $\dgg$ and $e\dgg$ corresponding to $\psi$ and $\tilde \psi$, respectively, then $\tilde \theta = \theta|_{e\dgg}$.
        \item[(b)] In the situation of (a), if $\theta$ is a ribbon structure on $\dgg$, then $\tilde \theta$ is a ribbon structure on $e\dgg$.
    \end{itemize}
\end{lemma}

The following Lemma gives us the structure of twists for easy unipotent groups:

\begin{lemma}[{\cite[Lemma 4.16]{characters}}]\label{easy twists}
    Let $G$ be an easy unipotent group over a field $k$ of characteristic $p>0$. Then, for every object $M \in \dgg$, its twist automorphism $\theta_M$ is trivial. 
\end{lemma}

In \Cref{2.2}, we consider the twist automorphism for unipotent groups with trivial $\L$-packets. 

\subsection{Sheaf-function correspondence}
Let $M \in \dgg$ such that $F^*M \cong M$, then we can define the ``trace of Frobenius" to be the map
$$t_M:G^F \longrightarrow  \qlb$$
$$g \mapsto \sum_i (-1)^i\mathrm{Tr}(F; H^i_c(G, M)_g).$$

We have that $t_M$ is a conjugation-invariant function of $G^F$. Note that we will denote the trace of Frobenius with respect to $F^m$ by $t_{M,m}$. Boyarchenko gives a relationship between the trace of Frobenius of $F$-stable character sheaves and the space $C(G^F/\sim)$ in \cite{chsheaves}, which is used to prove the ``only if" direction of \nameref{conj}. We also make use of this to prove the ``if" direction. Let $CS(G)^F$ denote the $F$-invariant character sheaves on $G$.

\begin{theorem}[{\cite[Theorem 1.8(b)]{chsheaves}}]\label{basis}
    Let $G$ be a connected unipotent group. The functions
    $$\{t_M : G^F \to \qlb \;| \;M\in CS(G)^F\} $$
    form a basis for the space $C(G^F/ \sim)$ which is orthonormal with respect to the inner product 
    $$\langle f_1 \; | \;f_2\rangle = \sum_{g \in G^F}f_1(g) \overline{f_2(g)}.$$
\end{theorem}

\section{Main Results}

\subsection{Asai twisting operator on easy algebraic groups}

For this section, let $G$ be a connected algebraic group over $\fqb$. Fix a Frobenius $F: G \rightarrow G$. Note that connectedness is required for the definition of the Asai twisting operator given in \Cref{asai def}.

\begin{lemma}\label{twisted cent}
    The Asai twisting operator is trivial on $C(G^F/ \sim)$ if and only if every $g \in G^F$ can be written as $z^{-1}F(z)$, where $z \in Z(g)$. 
\end{lemma}

\begin{proof}
    Since $G$ is connected, the Asai twisting operator is equal to $\mathrm{N}_{F/F}^*$. Thus, $\Theta_F$ is trivial if and only if the norm map
    $$\mathrm{N}_{F/F}:(G^F / \sim ) \rightarrow (G^F / \sim)$$ 
    $$g = xF(x)^{-1} \mapsto F(x)^{-1}x$$
    is constant on $G^{F}$-conjugacy classes. We also have that all $y\in G$ such that $y^{-1}F(x)^{-1}xy = xF(x)^{-1}$ are of the form $y=xz$, where $z \in Z(g)$. 
    
    Now suppose the Asai twisting operator is trivial. Then, $\mathrm{N}_{F/F}$ is constant on $G^{F}$-conjugacy classes, so we must have $zx \in G^{F}$ for some $z\in Z(g)$. Hence $xz=F(xz)$, so $z^{-1}F(z)=xF(x)^{-1}=g$.

    Conversely, suppose that for all $g\in G^F$, we have $g=z^{-1}F(z)$ for some $z \in Z(g)$. Then, $z^{-1}F(z)=xF(x)^{-1}$, which gives $xz=F(x)F(z)=F(xz)$. Hence, $xz \in G^F$ and $\mathrm{N}_{F/F}$ is constant on $G^F$ conjugacy classes.
\end{proof}

\begin{theorem}\label{asai prop}
   The Asai twisting operator is trivial on $C(G^{F^m}/\sim)$ for all $m \in \Z_{>0}$ if and only if $G$ is easy.
\end{theorem}
    
\begin{proof}
    By \Cref{twisted cent}, this is equivalent to proving that $G$ is easy if and only if for any $m \in \Z_{>0}$ every $g \in G^{F^m}$ can be written as $z^{-1}F^m(z)$ for some $z \in Z(g)^\circ$.

    $(\Rightarrow)$: If $G$ is easy, then for every $g \in G(k)$, we have $g \in Z(g)^\circ$. Then, by Lang's Theorem, $g=z^{-1}F^m(z)$ for some $z \in Z(g)$.

    $(\Leftarrow)$: We define $F^m$- twisted conjugation by $h$ on $g$ by $F^m(h)^{-1}gh$. Let $H^1(F^m, Z(g))$ denote the $F^m$- twisted conjugacy classes in $Z(g)$. By Lang's Theorem, we can write any $z_1 \in Z(g)^\circ$ as $h^{-1}F^m(h)$, with $h \in Z(g)^\circ$. Then for $z \in Z(g)$ and $z_1 \in Z(g)^\circ$, the map $zz_1 = zh^{-1}F^m(h) \mapsto h^{-1}zF^m(h)$ induces the equality $H^1(F^m, \pi_0(Z(g)))=H^1(F^m, Z(g))$.
    
    Suppose $g \in G^F$. By \Cref{twisted cent}, $g \sim 1 \in H^1(F^m,Z(g))=H^1(F^m, \pi_0(Z(g)))$ for all $m \in \Z_{>0}$. Since $\pi_0(Z(g))$ is finite, there exists an $m \in \Z_{>0}$ such that $F^m=\id$. Hence, $\bar{g}=\bar{1} \in \pi_0(Z(g))=Z(g)/Z(g)^\circ$, which gives us $g \in Z(g)^\circ$, as desired. To show this holds for $g \in G^{F^k}$, substitute $F$ with $F^k$.
\end{proof}

\begin{remark}
   Since all easy algebraic groups are connected, the connectedness assumption is only required for the ``only if" direction.
\end{remark}

Note that the proof of \Cref{asai prop} does not depend on the choice of $F$. This gives us the following corollary:

\begin{corollary}
    Suppose there exists a Frobenius $F$ such that $\Theta_{F^m}$ is trivial for all $m$. Then, for any Frobenius $F'$, $\Theta_{F'}$ is trivial. 
\end{corollary}

\subsection{Twists on unipotent groups with trivial $\L$-packets} \label{2.2}

From now on, let $G$ be a unipotent group over $\Fq$, equipped with an $\Fq$ rational structure defined by a Frobenius, $F:G \rightarrow G$. The proof of \Cref{thm} relies on the structure of twists in the category $\dgg$ and their relationship to the Asai twisting operator.  
\begin{lemma}\label{triv twists}
    If the $\L$-packet of character sheaves corresponding to $e$ is a singleton, then the twists in the Hecke subcategory $e\dgg$ are trivial.
\end{lemma}

\begin{proof}
    Let $M \in e\dgg$. By \Cref{edgg cat thm}(c), $M$ is a direct sum of shifts of character sheaves in the $\L$-packet corresponding to $e$. Since the $\L$-packets are trivial, $M$ is a direct sum of shifts of $e$. In particular, for all $i$, the cohomology sheaf $\HH^i(M)$ is either a direct sum of copies of $e$ or $0$. Since $\theta_e = \mathrm{id}_e$, we have $\theta_{\HH^i(M)}=\mathrm{id}_{\HH^i(M)}$ for all $i$. Thus, $\theta_M$ is a unipotent automorphism. Since $G$ is unipotent, $G$ has exponent $p^n$ for some $n \in \N$. Hence, $(\theta_M)^{p^n}=\mathrm{id}_M$. We then have $\theta_M=\mathrm{id}_M$.
\end{proof}

\begin{prop}[{\cite[Proposition 8.1(c)]{characters}}]\label{conn prop} 
    Let $e'_{\LL} \in \dgg$ be a minimal closed idempotent constructed from an admissible pair $\hl$ with $N_G\hl=G$. Let $\theta$ denote the twist automorphism of the identity in $\dgg$. Then, if the restriction of $\theta$ to the Hecke subcategory
    $$e'_\LL\dgg \subset \dgg$$
    is trivial, then $G$ is connected.
\end{prop}

\begin{lemma}\label{conn lem}
    If $G$ has trivial $\L$-packets, then $G$ is connected.
\end{lemma}

\begin{proof}
    By \Cref{ad pair prop}, $(G^\circ, \qlb)$ is an admissible pair for $G$ with $N_G(G^\circ, \qlb) =G$. Since $G$ has trivial $\L$-packets, the Hecke subcategory $e'_{\qlb}\dgg$ has trivial twists. By \Cref{conn prop}, $G$ is connected. 
\end{proof}

The proof of \Cref{thm} utilizes the relationship between twists in the category $\dgg$ and the Asai twisting operator that Deshpande explores in \cite{unipshint}. 

\begin{prop}[{\cite[Proposition 3.4]{unipshint}}]\label{eigen}
   Let $M \in \dgg$ be such that $F^*M\cong M$. 
   $$\Theta_F(t_M)=\theta^{-1}_Mt_M.$$
\end{prop}

Note that here we are we are identifying $\theta_M$ with a scalar via the isomorphism $\Hom(M,M) \cong \qlb^\times$.

\begin{remark}
   By \Cref{basis}, we have that the trace of Frobenius of the $F$-stable character sheaves form a basis for the space $C(G^F/\sim)$. Then, \Cref{eigen} tells us that these are eigenvectors for the twisting operator $\Theta_F$.
\end{remark}

We now have the tools to prove the following: 

\begin{theorem}\label{thm}
    If $G$ has trivial $\L$-packets, then $G$ is easy.
\end{theorem}
    
\begin{proof}
    By \Cref{conn lem}, $G$ is connected. Hence by Theorem 2.2, it is enough to show that $\Theta_{F^m}$ is trivial on $C(G^{F^m}/ \sim)$ for all $m \in \Z_{>0}.$ Recall that by \Cref{basis}, the functions $t_{M,m}$ for $F^m$-stable character sheaves $M$ form a basis of the vector space $C(G^{F^m}/\sim)$. Furthermore, by \Cref{eigen}, this basis is an eigenbasis for $\Theta_{F^m}$ with eigenvalues determined by the twisting operators $\theta_M$. Hence it is enough to show that $\theta_M= 1$ for every $F^m$-stable character sheaf $M$. By assumption, $G$ has trivial $\L$-packets, so by \Cref{edgg cat thm}(b), every character sheaf is a shift of some minimal closed idempotent $e$. Necessarily $\theta_e = 1$ since $e$ is a unit object with respect to convolution in $e\dgg$, so the proof is complete.
\end{proof}

Invoking \Cref{easy twists}, the proof on \Cref{thm} gives us the following corollary: 

\begin{corollary}
     The following are equivalent:
    \item[(1)] $G$ is easy;
    \item[(2)] $G$ has trivial $\L$-packets;
    \item[(3)] For all $M \in \dgg$, the twist automorphism $\theta_M$ is the identity.
\end{corollary}

\bibliography{refs.bib}

@article{foundations,
    author = "Mitya Boyarchenko and Vladimir Drinfeld",
    title = "Character sheves on unipotent groups in positive characteristic: foundations",
    journal = "Sel. Math. New Ser.",
    volume = "20",
    pages = "125--135",
    year = "2014",
    doi = "https://doi.org/10.1007/s00029-013-0133-7",
}

@article{characters,
    author = "Mitya Boyarchenko",
    title = "Characters of unipotent groups over finite fields",
    journal = "Sel. Math. New Ser.",
    volume = "16",
    pages = "857--933",
    year = "2010",
    
}

@article{chsheaves,
    author = "Mitya Boyarchenko",
    title = "Character sheaves and characters of unipotent groups over finite fields",
    journal = "Amer. J. Math.",
    volume = "135",
    number= "3",
    pages = "663--719",
    year = "2013",
}

@article{motivatedintro,
      title={A motivated introduction to character sheaves and the orbit method for unipotent groups in positive characteristic}, 
      author={Mitya Boyarchenko and Vladimir Drinfeld},
   
      eprint={math.RT/0609769},
      primaryClass={math.RT},
    
      url={https://arxiv.org/abs/math/0609769}, 
    year={2010},
}

@article{Lusztig2006,
    author = "George Lusztig",
    title = "Character sheaves and generalizations",
    journal = "The Unity of Mathematics: In Honor of the Ninetieth Birthday of I.M. Gelfand",
    volume = "17",
    number= "3",
    pages = "443--455",
    year = "2006",

}

@article{shintani,
title = "Shintani descent for algebraic groups over a finite field, I",
journal = "Journal of Algebra",
volume = "145",
number = "2",
pages = "468-524",
year = "1992",
issn = "0021-8693",
doi = "https://doi.org/10.1016/0021-8693(92)90113-Z",

author = "Toshiaki Shoji"
}

@article{unipshint, 
title="Shintani descent for algebraic groups and almost characters of unipotent groups", 
volume="152", 
DOI="10.1112/S0010437X16007429", 
number="8", 
journal="Compositio Mathematica", 
author="Tanmay Deshpande", 
year="2016", 
pages="1697–1724",

}

\end{document}